\documentclass[a4paper]{amsart}

\usepackage{amsthm}
\usepackage{amsfonts}
\usepackage{amsmath,amsthm,amssymb}
\usepackage{comment}
\usepackage{fancybox}
\usepackage{epsf}
\usepackage[dvips]{graphicx}

\newtheorem{theorem}{Theorem}[section]
\newtheorem{lemma}[theorem]{Lemma}

\theoremstyle{definition}

\theoremstyle{remark}
\newtheorem{remark}[theorem]{Remark}
\newtheorem{example}[theorem]{Example}

\numberwithin{equation}{section}

\begin{document} 

\title[A relation between Milnor's $\mu$-invariants and HOMFLYPT polynomials]{A relation between Milnor's $\mu$-invariants and HOMFLYPT polynomials}

\author[Yuka Kotorii]{Yuka Kotorii} 

\begin{abstract} 
Polyak showed that any Milnor's $\overline{\mu}$-invariant of length 3 
can be represented as a combination of the Conway polynomials of knots obtained 
by certain band sum of the link components.
On the other hand, Habegger and Lin showed that Milnor invariants are also invariants for string links, called $\mu$-invariants. 
We show that any Milnor's ${\mu}$-invariant of length $\leq k+2$ can be represented as a combination of the HOMFLYPT polynomials of knots obtained from the string link by some operation, if all ${\mu}$-invariants 
of length $\leq k$ vanish.  
Moreover, ${\mu}$-invariants of length $3$ are given by a combination of the Conway polynomials and linking numbers without any vanishing assumption.  
\end{abstract} 

\thanks{ A part of this work was supported by Platform for Dynamic Approaches to
Living System from the Ministry of Education, Culture, Sports, Science and
Technology, Japan.}


\maketitle

\section{Introduction}
For an ordered oriented link in the 3-sphere, J. Milnor \cite{Milnor, Milnor2} defined a family of invariants, 
known as {\em Milnor's $\overline{\mu}$-invariants}. 
For an $n$-component link $L$, Milnor invariant is determined by a sequence $I$ of elements in $\{1,2,\cdots ,n\}$
and denoted by $\overline{\mu}_L(I)$. 
It is known that Milnor invariants of length two are just linking numbers.
In general, Milnor invariant $\overline{\mu}_L(I)$ is only well-defined modulo the greatest common divisor $\Delta _L(I)$ 
of all Milnor invariants $\overline{\mu}_L(J)$ such that $J$ is a subsequence of $I$ obtained by removing at least one 
index or its cyclic permutation. 
If the sequence is of distinct numbers, then this invariant is also a link-homotopy invariant and 
we call it {\em Milnor's link-homotopy invariant}.
Here, the {\em link-homotopy} is an equivalence relation generated by ambient isotopy and self-crossing changes. 

In \cite{HL}, N. Habegger and X. S. Lin showed that Milnor invariants are also invariants for string links, and these invariants are called Milnor's $\mu$-invariants.
For any string link $\sigma$, $\mu_{\sigma}(I)$ coincides with $\bar{\mu}_{\hat{\sigma}}(I)$ modulo $\Delta_{\hat{\sigma}}(I)$, 
where $\hat{\sigma}$ is a link obtained by the closure of $\sigma$.
Milnor's $\mu$-invariants of length $k$ are finite type invariants of degree $k-1$ for any natural integer $k$, as shown by D. Bar-Natan \cite{BN2} and X. S. Lin \cite{Lin}.  

In \cite{P}, M. Polyak gave a formula expressing Milnor's $\bar{\mu}$-invariant of length 3 by the Conway polynomials of knots. 
His idea was derived from the following relation. 
Both Milnor's $\mu$-invariant of length 3 for string link and the second coefficient of the Conway polynomial are finite type invariants of degree 2. 
He gave this relation by using Gauss diagram formulas.  
 
Then, in \cite{MY}, J-B. Meilhan and A. Yasuhara generalized it by using the clasper theory introduced by K. Habiro \cite{H}. 
They showed that general Milnor's $\bar{\mu}$-invariants can be represented by the HOMFLYPT polynomials of knots under some assumption. 
Moreover the author and A. Yasuhara improved it in \cite{KY}.

In this paper, we give a formula expressing Milnor's $\mu$-invariant by the HOMFLYPT polynomials of knots under some assumption (Theorem~\ref{main}) by using the clasper theory in \cite{H}. The course of proof is similar to that in \cite{MY}. 
Moreover, Milnor's $\mu$-invariants of length 3 for any string link are given by the Conway polynomial, which is a finite type invariant of degree 2, and the linking number (Theorem~\ref{main2}).
It is a string link version of Polyak's result, and by taking modulo $\triangle{(I)}$, our result coincides with Polyak's result. \\

Given a sequence $I$ of elements in $\{1,2,\cdots ,n\}$,  
$J<I$ will be used for any subsequence $J$ of $I$, possibly $I$ itself, 
and $|J|$ will denote the length of the sequence $J$.

Let $\sigma$ be an $n$-string link. 
Given a sequence $I=i_1i_2 \ldots i_n$ obtained by permuting $12 \ldots n$ and a subsequence $J=j_1j_2 \ldots j_k$ of $I$,
we define a knot $\overline{\sigma_{I,J}}$ as the closure of the product $b_I \cdot \sigma_J$. 
Here $\sigma_J$ is the $n$-string link obtained from $\sigma $ by replacing the $i$-th string with the trivial string underpassing all other components for all $i \in \{ 1, 2, \cdots , n \} \setminus \{ j_1,j_2, \cdots ,j_k\}$,
and $b_I$ is the $n$-braid in $[0,1] \times [0,1]$ associated with the permutation 
$ \displaystyle
    \bigl(\begin{smallmatrix}
    i_1 & i_2 &  \ldots & i_{n-1} & i_{n} \\
    i_2 & i_3 &  \ldots &  i_{n}  &  i_1
  \end{smallmatrix}\bigr)$
such that the string connecting the $i_m$-th point on $[0,1] \times \{0\}$ with the $i_{m+1}$-th point on $[0,1] \times \{1\}$ ($1 \leq m < n$)   
underpasses all strings connecting  the $i_{m'}$-th point on $[0,1] \times \{0\}$ with  the $i_{m'+1}$-th point on $[0,1] \times \{1\}$ ($m < m' < n$) and the $i_{n}$-th point on $[0,1] \times \{0\}$ with the $i_{1}$-th point on $[0,1] \times \{1\}$. 
See Figure \ref{stringlink1} for an example.
We then have the following Theorem. 

\begin{theorem}\label{main}
Let $\sigma$ be an $n$-string link ($n \geq 4$) with vanishing Milnor's link-homotopy 
invariants of length $\leq n-2$. 
Then for any sequence $I$ obtained by permuting $12 \ldots n$, 
we have 
\[ {\mu}_{\sigma}(I) = \frac{(-1)^{n-1}}{2^{n-1}(n-1)!} \sum_{J<I} (-1)^{|J|}P_0^{(n-1)}(\overline{\sigma_{I,J}} ;1), \]
where $P_0^{(n-1)}( \ \cdot \  ;1)$ is the $(n-1)$-th derivative of the 0-th coefficient polynomial $P_0( \ \cdot \ ; t)$ of the HOMFLYPT polynomial $P_0( \ \cdot \  ; t, z)$ evaluated at $t=1$. 
\end{theorem}

Note that the above vanishing assumption for a string link is equivalent to that any ($n-2$)-substring link is link-homotopic to the trivial string link.  \\

We also give the case of $\mu$-invariants of length 3 without the assumption. 

\medskip

\begin{theorem}\label{main2}
Let $\sigma$ be a $3$-string link and $I=i_1i_2i_3$ be a sequence obtained by permuting $123$.
We then have 
\begin{align*}
{\mu}_{\sigma}(I) = - \sum_{J<I} (-1)^{|J|} a_{2}(\overline{\sigma_{I,J}}) -lk_\sigma(i_1i_2)lk_\sigma(i_2i_3) + A_{I},
 \end{align*}
where $a_2$ is the second coefficient of the Conway polynomial, 
$lk_\sigma (ij)$ is the linking number of the $i$-th component and  $j$-th component of $\sigma$, and  
\begin{align*}
A_{I} = 
\begin{cases}
    lk_\sigma(i_1i_2)  & (I=312) \\ 
    -lk_\sigma(i_1i_2) & (I=132) \\ 
    0 & (otherwise). 
\end{cases}
 \end{align*}
\end{theorem}

\begin{remark}
For any $n$-string link, we can give essentially the same results as Theorem~\ref{main} and Theorem~\ref{main2} for  $\mu$-invariant of length less than $n$ by ignoring $i$-th strings such that $i$ does not apper in the sequence $I$,
because we have that $\mu_\sigma(I)= \mu_{\cup \sigma_i} (I)$, where  $\cup\sigma_i$ means a disjoint union of $\sigma_i$ such that $i$ appears in the sequence $I$.
\end{remark}

\begin{remark}
This operation from a string link to a knot corresponds to $Y$-graph sum of links defined by M. Polyak.
By taking this formula modulo $\Delta _{\overline{\sigma_{I,J}}}(I)$, we get Polyak's relation between Milnor's $\overline{\mu}$-invariants and Conway polynomials \cite{P}. 
\end{remark}

\begin{remark}
K. Taniyama gave a formula expressing Milnor's $\overline{\mu}$-invariants of length 3 for links by the second coefficient of the Conway polynomial assuming that all linking numbers vanish in \cite{T}.
\end{remark}

\begin{remark}
In \cite{M}, J.B. Meilhan showed that all finite type invariants of degree 2 for string links was given a formula by some invariants (Theorem 2.8).
So the formula in Theorem~\ref{main2} could also be derived from \cite{M}.  
\end{remark}

\section*{Acknowledgements} 
The author thanks Professor Sadayoshi Kojima for comments and suggestions.
She also thanks Professor Akira Yasuhara for discussions and comments. 
She also thanks Professor Michael Polyak for valuable advices. 
She also thanks Professor Jean-Baptiste Meilhan for many useful comments. 
She also thanks Professor Kouki Taniyama for comments. 
\section{Some known results}

\medskip
\subsection{String link} 
Let $n$ be a positive integer and $D^2 \subset  \mathbb{R}^2$ the unit disk equipped with $n$ marked points $x_1, x_2, \cdots, x_n$ in its interior, lying in the diameter on the $x$-axis of  $\mathbb{R}^2$ as illustrated in Figure \ref{stringlink}.  
Let $I=[0,1]$. 
An {\em $n$-string link} $\sigma$ is the image of a proper embedding 
$\sqcup_{i=1}^n I_i \rightarrow D^2 \times I$ of the disjoint union of $n$ copies of $I$ in $D^2 \times I$, 
such that $\sigma|_{I_i}(0)=(x_i, 0)$ and $\sigma|_{I_i}(1)=(x_i, 1)$ for each $i$ as illustrated in Figure \ref{stringlink}. 
Each string of a string link inherits an orientation from the usual orientation of $I$. 
The $n$-string link $\{x_1,x_2,\cdots,x_n\}\times I$ in $D^2 \times I$ is called the {\em trivial $n$-string link} 
and denoted by ${\bf 1}_n$ or ${\bf 1}$ simply.

\begin{figure}[h]
  \begin{center}
\includegraphics[width=.6\linewidth]{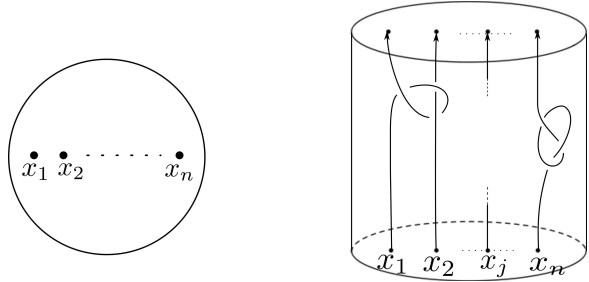}
    \caption{An n-string link}
    \label{stringlink}
  \end{center}
\end{figure}

Given two $n$-string links $\sigma$ and $\sigma'$, we denote their product by $\sigma \cdot \sigma'$, 
which is given by stacking $\sigma'$ on the top of $\sigma$ and reparametrizing the ambient cylinder $D^2 \times I$. 
By this product, the set of isotopy classes of $n$-string links has a monoid structure with unit given by the trivial string link ${\bf 1}_n$.  
Moreover, the set of link-homotopy classes of $n$-string links is a group under this product. \\

\subsection{Claspers}
The theory of claspers was introduced by K. Habiro \cite{H}.  
Here, we define only simple tree clasper. 
For a general definition, we refer the reader to \cite{H}.  

Let $L$ be a (string) link.  
A disk $T$ embedded in $S^3$ (or $D^2\times I$) is called a {\em simple tree clasper} (we will call it {\em tree clasper} or {\em tree}, simply in this paper) for $L$ if it 
satisfies the following four conditions:
\begin{enumerate} 
\item The disk $T$ is decomposed into disks and bands.
Here, the band connects two distinct disks, and are called {\it edges}. 
\item The disks attach either 1 or 3 edges. We call a disk attached with only one edge a {\em leaf}.
\item The disk $T$ intersects the (string) link $L$ transversely and the intersections are contained in the interiors of the leaves. 
\item Each of leaves of $T$ intersects $L$ at exactly one point.   
\end{enumerate}
A simple tree clasper $T$ with $k+1$ leaves is called a tree clasper of {\it degree} $k$ or \emph{$C_k$-tree}.

Given a $C_k$-tree $T$ for a (string) link $L$, there exists a procedure to construct a
framed link $\gamma(T)$ in a regular neighborhood of $T$. 
We call surgery along $\gamma(T)$ {\em surgery along} $T$. 
Because there is an orientation-preserving homeomorphism, fixing the boundary, 
from the regular neighborhood $N(T)$ of $T$ to the manifold obtained from $N(T)$ by surgery along $T$, 
the surgery along $T$ can be regarded as a local move on $L$. 
We denote by $L_T$ a (string) link obtained from $L$ by surgery along $T$. 
For example, surgery along a $C_k$-tree is a local move as illustrated in Figure~\ref{Ck-tree}.
In this paper, the drawing convention for $C_k$-trees are those of \cite[Figure 7]{H}.
Similarly, let $T_1 \cup \cdots \cup T_m$ be a disjoint union of trees for $L$, 
we can define $L_{T_1 \cup \cdots \cup T_m}$ as the link obtained by surgery along $T_1 \cup \cdots \cup T_m$. 
 
\begin{figure}[h]
  \begin{center}
\includegraphics[width=.8\linewidth]{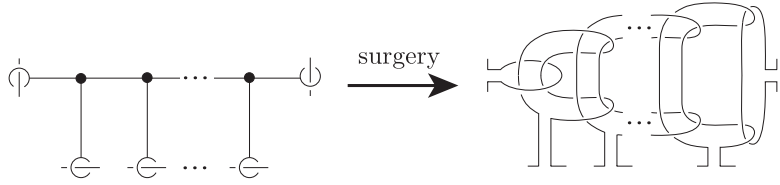}
    \caption{Surgery along a simple tree clasper}
    \label{Ck-tree}
  \end{center}
\end{figure}

The $C_k$-equivalence is an equivalence relation on (string) links generated by surgeries along $C_k$-tree claspers and ambient isotopy.

\subsection{Milnor's $\mu$-invariant for string links}
Let $\sigma=\cup^n_{i=1} \sigma_i$ in $D^2 \times I$ be an $n$-string link.
We consider the fundamental group $\pi_1(D^2 \times I \setminus \sigma)$ of the complement of $\sigma$ in $D^2 \times I$, where we choose a point $b$ as  a base point and curves $\alpha_1, \cdots , \alpha_n$ as meridians in Figure \ref{longi}.

\begin{figure}[h]
  \begin{center}
\includegraphics[width=.6\linewidth]{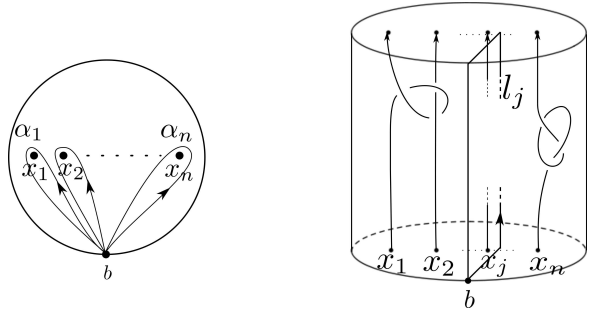}
    \caption{Longitude of string link}
    \label{longi}
  \end{center}
\end{figure}

By Stallings' theorem \cite{S}, for any positive integer $q$, the inclusion map
\[ \iota : D^2 \times \{0\} \setminus \{ x_1, \cdots , x_n \} \longrightarrow  D^2 \times I \setminus \sigma \]
induces an isomorphism of the lower central series quotients of the fundamental groups    
\[ \iota_* : \frac{ \pi_1(D^2 \times \{0\} \setminus \{ x_1, \cdots , x_n \}) }{ ( \pi_1(D^2 \times \{0\} \setminus \{ x_1, \cdots , x_n \}) )_q } \longrightarrow  \frac{ \pi_1(D^2 \times I \setminus \sigma) }{ \pi_1(D^2 \times I \setminus \sigma)_q}, \]
where given a group $G$, $G_q$ means the $q$-th lower central subgroup of  $G$. 
The fundamental group $\pi_1(D^2 \times \{0\} \setminus \{ x_1, \cdots , x_n \})$ is a free group generated by $\alpha_1, \cdots, \alpha_n$.
We then consider the $j$-th longitude $l_j$ of $\sigma$ in $D^2 \times I$, where $l_j$ is the closure of the preferred parallel curve of $\sigma_j$, whose endpoints lie on the $x$-axis in $D^2 \times \{0,1\} $ as in Figure \ref{longi}.
We then consider the image of the longitude $\iota_*^{-1}(l_j)$ by the Magnus expansion and denote $\mu(i_1, \cdots, i_k,j)$ the coefficient of $X_{i_1}X_{i_2} \cdots X_{i_k}$ in the Magnus expansion.

\begin{theorem}[\cite{HL}]
For any positive integer $q$, if $k<q$, then $\mu(i_1, \cdots, i_k,j)$ is invariant under isotopy.
Moreover, if the sequence $i_1, \cdots, i_k,j$ is of distinct numbers, then $\mu(i_1, \cdots, i_k,j)$ is also link-homotopy invariant.
\end{theorem} 
We call this invariant {\it Milnor's $\mu$-invariant}.  
The following theorem and lemma play important roles in calculating Milnor's invariants.

\begin{theorem}[{\cite[Theorem 7.2]{H}}] \label{MilnorC_k} 
The Milnor's invariants of length less than or equal to $k$ for (string) links are invariants of $C_k$-equivalence. 
\end{theorem}

\begin{lemma}[{\cite[Lemma 3.3]{MY1}}] \label{additivity} 
Let $\sigma$ and $\sigma'$ be $n$-string links. 
Let $h$ and $h'$ be the integers. 
If $\mu_{\sigma}(I)=0$ for any $I$ with $|I| \leq h$ 
and $\mu_{\sigma'}(I')=0$ for any $I'$ with $|I'| \leq h'$, 
then for any $J$ with $|J| \leq h+h'$
\[ \mu_{\sigma \cdot \sigma'}(J)= \mu_{\sigma}(J) + \mu_{\sigma'}(J). \\ \]  
\end{lemma}

\subsection{HOMFLYPT polynomial}


The \emph{HOMFLYPT polynomial} $P(L;t,z)\in {\Bbb Z}[t^{\pm 1},z^{\pm 1}]$ of an oriented link $L$ 
is defined by the following two formulas:
\begin{enumerate}
\item $P(U;t,z) = 1$, and 
\item $t^{-1}P(L_+;t,z) - tP(L_- ;t,z) = zP(L_0 ;t,z)$, 
\end{enumerate}
where $U$ denotes the trivial knot and $L_+$, $L_-$ and $L_0$ are link diagrams which are identical everywhere except near one crossing, where they look as follows:
\[L_+=\begin{array}{c}
\includegraphics[width=.07\linewidth]{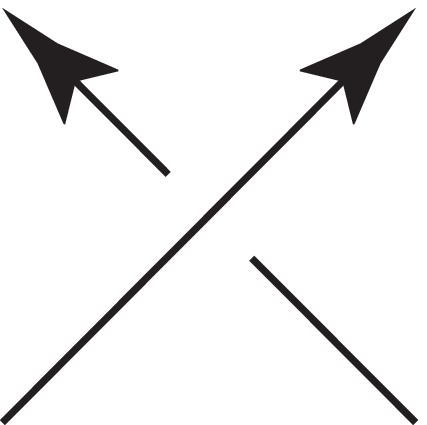}\end{array}~~;~~~~~~
L_-=\begin{array}{c}
\includegraphics[width=.07\linewidth]{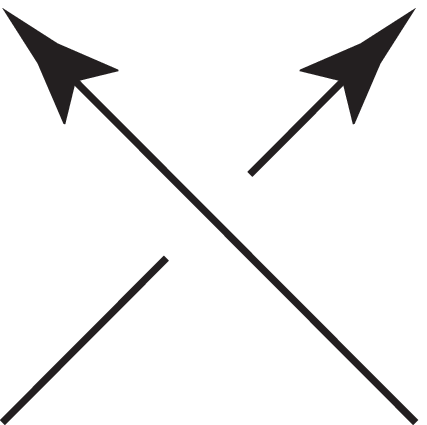}\end{array}~~;~~~~~~
L_0=\begin{array}{c}
\includegraphics[width=.07\linewidth]{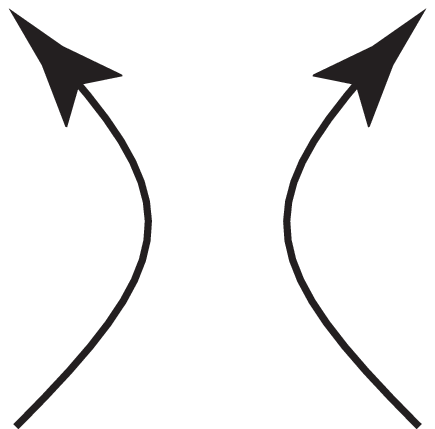}\end{array}~~.
\]

Recall that the HOMFLYPT polynomial of a knot $K$ is of the form $P(K;t,z)=\sum_{k=0}^{N} P_{2k}(K;t)z^{2k}$, where $P_{2k}(K;t)\in {\Bbb Z}[t^{\pm 1}]$ is called the $2k$-th coefficient polynomial of $K$.  

%

It is known that the HOMFLYPT polynomial of knots is multiplicative under connected sum. So $P_0(K;t)$ is also multiplicative under connected sum. 
For any knots $K$ and $K'$ and any integer $n$, we have 
\[ P_0^{(n)}(K \sharp K';1) = P_0^{(n)}(K;1) + P_0^{(n)}(K';1) + \sum_{k=2}^{n-2} \binom{n}{k} P_0^{(k)}(K;1) P_0^{(n-k)}(K';1), \]
because $P_0(K ;1)=1$ and $P_0^{(1)}(K ;1)=0$ for any knot $K$. 
Moreover, if the knot $K'$ is $C_{n-1}$-equivalent to the trivial knot, we have 
\begin{align}
 P_0^{(n)}(K \sharp K';1) =  P_0^{(n)}(K;1) + P_0^{(n)}(K';1), \label{connectedsum}
 \end{align}
because finite type invariants of degree $ < n-1$ are $C_{n-1}$-equivalence  invariants.

\section{Proof of Theorem~\ref{main} and Theorem~\ref{main2}}

\subsection{Preparation} 

Let $I=i_1i_2 \ldots i_n$ be a sequence obtained by permuting $ 1 2 \ldots n $.
Define the order $<_I$ on the integers from 1 to $n$ such that $i_1 <_I i_2 <_I \cdots <_I i_n$.
Let $M : \{ 1, 2, \cdots , k\} \longrightarrow  \{ 1, 2, \cdots , n\}$ ($k \leq n$) be an injection such that $M(1) <_I M(j) <_I M(k)$ ($j \in \{ 2, \cdots , k-1\}$) and let $\mathcal{M}_{I,k}$ be the set of such injections. Consider an element of $\mathcal{M}_{I,k}$ as a sequence of length $k$.

For any $M \in \mathcal{M}_{I,k}$, let $T_M$ and $T_M^{-1}$ be $C_{k-1}$-trees as illustrated in Figure \ref{C_k-trees}.
Here, Figure \ref{C_k-trees} are the images of homeomorphisms from the neighborhoods of $T_M$ and $T_M^{-1}$ to the 3-balls and $\oplus$ means a positive half-twist.
Although $({\bf 1}_n )_{T_M}$ and $({\bf 1}_n )_{T_M^{-1}}$ are not unique up to ambient isotopy, by {\cite[Lemma 2.1, 2.4]{Y}}, it is unique up to $C_k$-equivalence.
Therefore for any $M \in \mathcal{M}_{I,k}$, we may choose $({\bf 1}_n )_{T_M}$ and $({\bf 1}_n )_{T_M^{-1}}$ uniquely up to $C_k$-equivalence.
In particular, we may choose $({\bf 1}_n )_{T_M}$ and $({\bf 1}_n )_{T_M^{-1}}$ so that each component of ${\bf 1}_n$ underpasses all edges of the trees.

\begin{figure}[h]
  \begin{center}
\includegraphics[width=1.\linewidth]{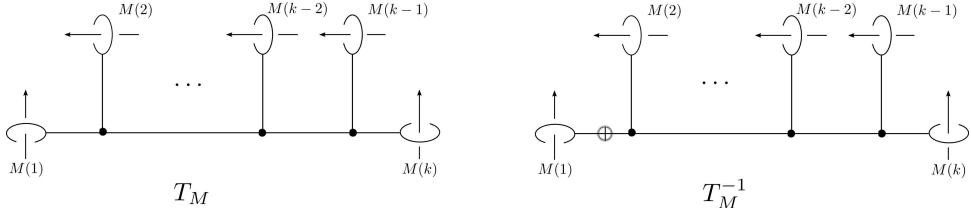}
    \caption{$T_M$ and $T_M^{-1}$}
    \label{C_k-trees}
  \end{center}
\end{figure}

%

We have the following lemma proved by the same method of {\cite[Lemma 4.2, Theorem 4.3]{Y}}.

\begin{lemma}[cf. {\cite[Theorem 4.3]{Y}}, {\cite[Theorem 4.1]{MY}}] \label{standerd form}
Let $\sigma$ be an $n$-string link and $I=i_1i_2 \ldots i_n$ be a sequence obtained by permuting $ 1 2 \ldots n $.
Then $\sigma$ is link-homotopic to a string link $l_1 \cdot  l_2 \cdot \cdots \cdot l_{n-1}$, where 
\[
  l_k= \prod_{M\in \mathcal{M}_{I,k+1}} {({\bf 1}_{T_M^{\epsilon}})}^{|x_M|}, 
    x_M =\left\{ 
          \begin{array}{ll}
               \mu_\sigma(M) & \text{ if } k=1 \\
                 \mu_\sigma(M)-\mu_{l_1 l_2\cdots l_{k-1}}(M)  & \text{ if } k \geq 2,  
          \end{array} 
            \right. 
\]            
where $M$ in the product appears in the lexicographic order of $\mathcal{M}_{I,k+1}$ for the order $<_I$, and $\epsilon=1$ if $x_M >0$ and otherwise $\epsilon=-1$.
\end{lemma}

\begin{remark}\label{MYrem}
In \cite{MY}, the statement slightly differs from the original one in \cite{Y}, 
because the authors used a different definition for the $C_k$-trees $T_M$ and $T_M^{-1}$ from {\cite{Y}}. But in {\cite[Theorem 4.1]{MY}}, when $i\geq 2$, a sign $(-1)^{i+1}$ of  $\mu_L(M)-\mu_{l_1 l_2\cdots l_{i-1}}(M)$ seems to be missing.
In this paper, we give $+1$ as a sign, by changing a part of orientations of a link (compare Figure~\ref{C_k-trees} with {\cite[Figure 4.1]{Y} and \cite[Figure 7]{MY}}).
\end{remark}

\subsection{Proof of Theorem~\ref{main}}

\begin{proof}[Proof of Theorem \ref{main}] 
The course of proof of Theorem~\ref{main} is similar to one in \cite[Theorem 1.1]{MY}. 
Let $n \geq 4$ and $\sigma$ be an $n$-string link with vanishing Milnor's link-homotopy invariants of length $\leq n-2$.
Let $I=i_1i_2 \ldots i_n$ be a sequence obtained by permuting $ 1 2 \ldots n $. 

In the proof, we will simply write $\overline{\sigma_{J}}$ for $\overline{\sigma_{I,J}}$. 
Moreover, concepts {\it in good position} and {\it weight} appearing in this proof are defined in \cite{MY}.  \\

\noindent Step 1: 
By combination of Lemma~\ref{standerd form} and the assumption that Milnor's  
link-homotopy invariants of length $\leq n-2$ vanish,
$\sigma$ is link-homotopic to $l_{n-2} \cdot l_{n-1}$, where $l_{n-2}=\prod_{M \in \mathcal{M}_{I,n-1}} {({\bf 1}_{T_M^{\epsilon}})}^{|x_M|}$ and $l_{n-1}=\prod_{M \in \mathcal{M}_{I,n}} {({\bf 1}_{T_M^{\epsilon}})}^{|x_M|}$.
Set   
\[\mathcal{F} := \left(\bigcup _{M \in \mathcal{M}_{I,n-1}}   {(T_M^{\epsilon})}^{|x_M|} \right) \cup \left(\bigcup _{M \in \mathcal{M}_{I,n}}   {(T_M^{\epsilon})}^{|x_{M}|} \right). \]
Then $l_{n-2} \cdot l_{n-1} $ can be regarded as the string link ${\bf 1}_{\mathcal{F} }$ obtained from ${{\bf 1}_n }$ by surgery along the union of trees $\mathcal{F}$.
Moreover, there is a disjoint union $\mathcal{R}_1$ of $C_1$-trees whose tree intersects a single component of the string link $l_{n-2} \cdot l_{n-1} $ such that 
\begin{align*}
    \sigma = (l_{n-2} \cdot  l_{n-1} )_{\mathcal{R}_1}=({{\bf 1} }_{\mathcal{F} })_{\mathcal{R}_1}.
\end{align*}

\noindent
Then $\sigma$ can be regarded as the string link ${\bf 1}_{\mathcal{F} \cup \mathcal{R}_1}$ obtained from ${\bf 1}_n$ by surgery along the union of trees $\mathcal{F} \cup \mathcal{R}_1$. 

In the canonical diagram of ${\bf 1}_n$, a tree for ${\bf 1}_n$ is {\em in good position} if each component of ${\bf 1}_n$ underpasses all edges of the tree.
Note that each tree of $\mathcal{F}$ is in good position, 
by the definition of $({\bf 1}_n)_{T_M}$ and ${({\bf 1}_n)_{T_M^{-1}}}$.
On the other hand,  a tree of $\mathcal{R}_1$ may not be in good position.  
We now replace $\mathcal{R}_1$ with some trees in good position up to $C_n$-equivalence, 
by repeated applications of \cite[Proposition 4.5]{H}, and we then have ${\bf 1}_{\mathcal{F} \cup \mathcal{R}_1} $ is $C_n$-equivalent to ${\bf 1}_{\mathcal{F} \cup \mathcal{R}}$, where $\mathcal{R}$ is a disjoint union of $C_k$-trees ($k \geq 2$) for ${\bf 1}_n$ in good position whose tree intersects some component of ${\bf 1}_n$ more than once. 

It follows from \cite[Lemma 3.2]{MY} that for any subsequence $J$ of $I$ (possibly $I$),  
\[
    P_0 \left( \overline{({{\bf 1}_{\mathcal{F} \cup \mathcal{R}})}_J} ; t \right) = P_0 \left( \overline{({{\bf 1}_{\mathcal{G} \cup \mathcal{R}})}_J} ; t \right),
\]
where 
\[
    {\mathcal{G}} = \left(\bigcup _{M \in \mathcal{M}_{I,n-1}, M<I} {(T_M^{\epsilon})}^{|x_M|} \right) \cup {(T_I^{\epsilon})}^{|x_I|}, 
\]  
meaning that for any $M \in \mathcal{M}_{I, n-1} \cup \mathcal{M}_{I, n}$, we can ignore $T_M$'s and $T_M^{-1}$'s such that $M$ is not subsequence of $I$ when computing $P_0( \ \cdot \ ; t)$. \\

\noindent Step 2:
By using the arguments of \cite{MY} for transformations of tree claspers, 
the knot $\overline{({{\bf 1}_{\mathcal{G} \cup \mathcal{R}})}_J}$ is $C_n$-equivalent to the connected sum
\begin{align*}
\sharp_{|x_I|} U_{(T_I^\epsilon)_J } \sharp_{M \in{\mathcal{M}_{I,n-1}}, M<J} ( \sharp_{|x_M|} U_{T_M^\epsilon } ) 
\sharp U_{(\mathcal{R}')_J}, 
\end{align*}
where $U_{(T_I^\epsilon)_J }$ means $\overline{({{\bf 1}_{T_I^\epsilon }})_J}$,  $U_{T_M^\epsilon }$ means $\overline{({{\bf 1}_{T_M^\epsilon }})_J}$, and $\mathcal{R}'$ is a disjoint union of some trees for the trivial knot $U$ such that the number of the weight of each tree is less than or equal to the degree and $(\mathcal{R}')_J$ is the union of trees of $\mathcal{R}'$ whose weight is a subset of the set of elements in $J$. 
By \cite{K}, ${ P_0^{(n-1)} (\ \cdot \ ;1) }$ is a finite type invariant of degree $n-1$.
Because a finite type invariant of degree $n-1$ is an invariant of $C_n$-equivalence \cite{H},  
${ P_0^{(n-1)} (\ \cdot \ ;1) }$ is invariant under $C_n$-equivalence.
Moreover by using Equation (\ref{connectedsum}), for any subsequence $J$ of $I$, we have 
\[ \begin{array}{rcl}
\displaystyle\sum_{J<I}(-1)^{|J|} { P_0^{(n-1)} \left( \overline{\sigma_J};1 \right) }
&=& \displaystyle\sum_{J<I}(-1)^{|J|} { P_0^{(n-1)} (\overline{({\bf 1}_{\mathcal{G} \cup \mathcal{R}})_J};1) } \\
&=& (-1)^{|I|} |x_I| { P_0^{(n-1)}(U_{(T_I^\epsilon)_I } ;1)} \\
&&\hspace{-.4cm} + \displaystyle \sum_{J<I}
(-1)^{|J|} \Large{ \displaystyle\sum_{M<J, M \in \mathcal{M}_{I, n-1}} |x_{M}| { P_0^{(n-1)} (U_{T_M^\epsilon };1)} \Large} \\
&&\hspace{-.4cm} + \displaystyle\sum_{J<I}(-1)^{|J|} { P_0^{(n-1)} (U_{(\mathcal{R}')_J};1). } \\
\end{array}
\]
By using the method of \cite{MY}, we have
\[ \displaystyle\sum_{J<I}(-1)^{|J|} { P_0^{(n-1)} \left( \overline{\sigma_J};1 \right) }= (-1)^{n-1} x_I 2^{n-1} (n-1)!. \]
On the other hand, 
by Lemma \ref{additivity}, $\mu_{l_{n-2}} (I)=0$ for $n \geq 4$.
Thus, by the definition of $x_I$ we have 

\begin{align*}\label{claim:milnor1}
x_I =  {\mu}_{\sigma}(I) - \mu_{l_{n-2}} (I)  =  {\mu}_{\sigma}(I). 
\end{align*}
Therefore, we have
\begin{align*}
\mu_{\sigma}(I)= x_I=   
 &  \frac{(-1)^{n-1}}{2^{n-1}(n-1)!} \sum_{J<I} (-1)^{|J|}  { P_0^{(n-1)}(\overline{\sigma_J} ;1)}.
\end{align*} 
\end{proof} 

\subsection{Proof of Theorem~\ref{main2}}

We consider the case $n=3$.
This proof is similar to that of Theorem~\ref{main}.
However, in this case, when we transform $\overline{\sigma_J}$ to the connected sum of some knots up to $C_3$-equivalence,
a form of the connected sum is affected by new trees which come from leaf slides.
We note that $P_0^{(2)}=-8a_2$.

\begin{proof}[Proof of Theorem \ref{main2}]

Let $\sigma$ be a 3-string link and $I=i_1i_2i_3$ a sequence obtained by permuting 123. 
Similar to Step 1 of Theorem~\ref{main}, we have that for any subsequence $J$ of $I$ (possibly $I$), 
$$ P_0 \left({\overline{\sigma_J}}; t \right) =P_0 \left( \overline{{({\bf 1}_{\mathcal{G} \cup \mathcal{R}})_J}}; t \right),$$
where 
\[
    {\mathcal{G}} =  (T_{i_1i_2}^{\epsilon})^{|x_{i_1i_2}|} \cup  (T_{i_1i_3}^{\epsilon})^{|x_{i_1i_3}|} \cup (T_{i_2i_3}^{\epsilon})^{|x_{i_2i_3}|}  \cup {(T_I^{\epsilon})}^{|x_I|} 
\]  
and $\mathcal{R}$ is a disjoint union of trees for ${\bf 1}_n$ in good position whose tree intersects some component of ${\bf 1}_n$ more than once. 

By {\cite[Lemma 2.2]{MY}} and {\cite[Lemma 2.4]{Y}}, 
we have that $\overline{\sigma_I}$ is $C_3$-equivalent to the connected sum
\begin{align*}
\sharp_{|x_I|} U_{T_I^\epsilon } \sharp U_{\mathcal{G}'} \sharp U_{\mathcal{R}'} 
\end{align*}
where $U_{T_I^\epsilon }$ means $\overline{{({\bf 1}_{T_I^\epsilon })}_I}$, 
$U_{\mathcal{G}'}$ means $\overline{({\bf 1}_{\mathcal{G}'})_I}$ 
where ${\mathcal{G}'} =  (T_{i_1i_2}^{\epsilon})^{|x_{i_1i_2}|} \cup  (T_{i_1i_3}^{\epsilon})^{|x_{i_1i_3}|} \cup (T_{i_2i_3}^{\epsilon})^{|x_{i_2i_3}|}$ and $\mathcal{R}'$ is a disjoint union of some trees for $U$ such that the number of the weight is less than or equal to the degree. 
Since $a_2$ is additive for the connected sum of knots, we have 
\begin{align*}
 \begin{array}{rcl}
\displaystyle \sum_{J<I}(-1)^{|J|}  a_2({\overline{\sigma_J}}) 
&=& \displaystyle\sum_{J<I}(-1)^{|J|}a_2(\overline{{({\bf 1}_{\mathcal{G} \cup \mathcal{R}})_J}}) \\
&=& (-1)^{|I|} \{ |x_{I}|  a_2(U_{T_{I}^\epsilon }) + a_2(U_{\mathcal{G}'}) \} \\
&=& - (x_{I} + x_{i_1i_2}x_{i_1i_3} + x_{i_1i_3}x_{i_2i_3} ). \\ 
 \end{array}
\end{align*}

On the other hand, by the definition of $x_{M}$ ($M=i_1i_2, i_1i_3, i_2i_3$) and $x_{I}$ in Lemma~\ref{standerd form}, we have that
$x_{M} = lk_\sigma(M)$
and $x_{I} =  \mu_\sigma(I) -  \mu_{l_1}(I)$, 
where $l_1={({\bf 1}_{T_{i_1i_2}^{\epsilon }})}^{|x_{i_1i_2}|} \cdot {({\bf 1}_{T_{i_1i_3}^{\epsilon }})}^{|x_{i_1i_3}|} \cdot {({\bf 1}_{T_{i_2i_3}^{\epsilon }})}^{|x_{i_2i_3}|}$. 
By the definition of the Milnor's $\mu$-invariants (see \cite{Milnor, Milnor2} for details to calculate), we have
$$ \mu_{l_1}(I) =  lk_\sigma(i_1i_2) lk_\sigma(i_1i_3) + lk_\sigma(i_1i_3)lk_\sigma(i_2i_3)  - lk_\sigma(i_1i_2)lk_\sigma(i_2i_3) + A_I ,$$
where 
\begin{align*}
A_{I} = 
\begin{cases}
    lk_\sigma(i_1i_2)  & (I=312) \\ 
    -lk_\sigma(i_1i_2) & (I=132) \\ 
    0 & (otherwise). 
\end{cases}
 \end{align*}
Therefore, we have 
 the formulas in Theorem~\ref{main2}. 
\end{proof}

\section{Examples}

\begin{example}
Let $\sigma$ be a 3-string link showed by Figure \ref{stringlink1}. 
Then $\mu_{123}(\sigma)=-1$, $\mu_{132}(\sigma)=\mu_{213}(\sigma)=1$ and $\mu_{231}(\sigma)=\mu_{312}(\sigma)=\mu_{321}(\sigma)=0$. 
And $lk_\sigma(12)=lk_\sigma(23)=1$ and $lk_\sigma(13)=0$.
 
On the other hand, $\overline{\sigma_{123,123}}$ and $\overline{\sigma_{123, 23}}$ are the figure-eight knot, and $\overline{\sigma_{123,J}}$ ($J \neq 123, 23$) is the trivial knot.
Therefore we obtain 
\begin{align*}
- \sum_{J<123} (-1)^{|J|} a_2(\overline{\sigma_{123,J}}) -lk_\sigma(12) lk_\sigma(23) 
 = a_2(4_1) - a_2(4_1) - 1 \cdot 1 
= -1. 
\end{align*}
Similarly, we have equations between Milnor's invariants and Conway polynomials.

\begin{figure}[h]
  \begin{center}
\includegraphics[width=.81\linewidth]{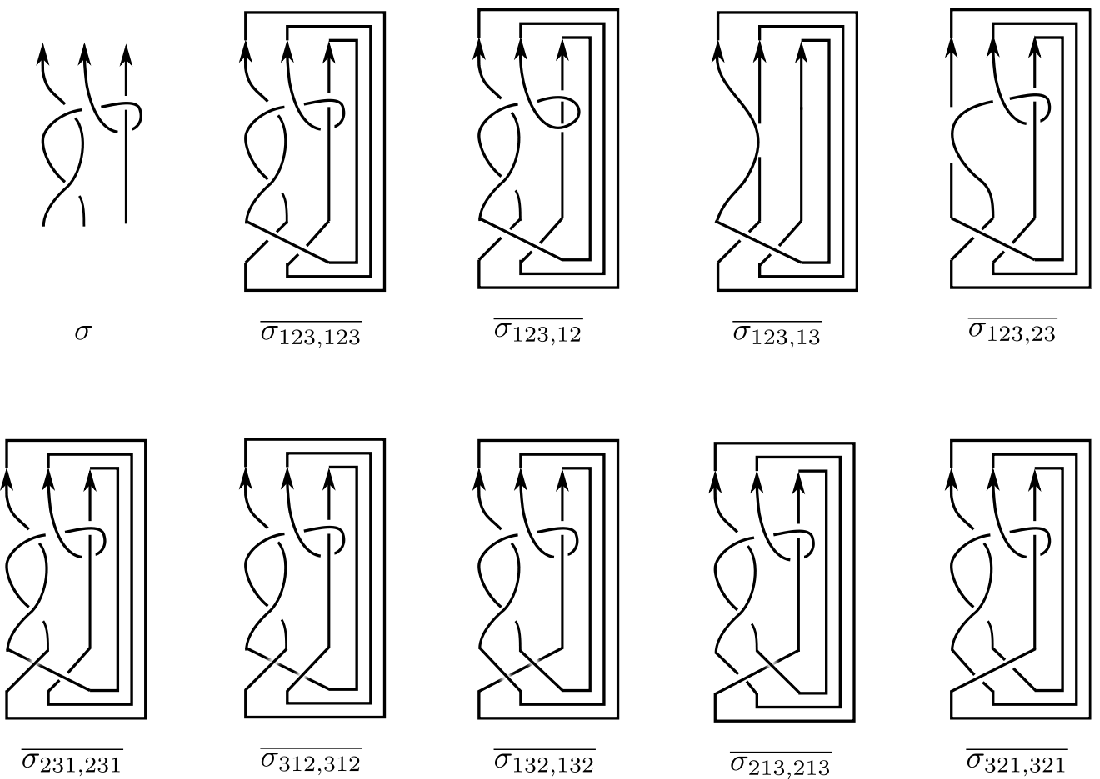}
     \caption{Example}
    \label{stringlink1}
  \end{center}
\end{figure}
\end{example}

%


\end{document}